\documentclass[10pt]{amsart}
\usepackage{amsmath,amssymb,amsthm,needspace}
\usepackage[hidelinks]{hyperref}
\hypersetup{pdftitle={When Hilbert functions determine generic initial ideals},pdfauthor={Van Duc Trung}}
\numberwithin{equation}{section}
\newtheorem{theorem}{Theorem}[section]
\newtheorem{introtheorem}{Theorem}

\newtheorem{lemma}[theorem]{Lemma}
\newtheorem{proposition}[theorem]{Proposition}
\newtheorem{corollary}[theorem]{Corollary}
\theoremstyle{definition}

\theoremstyle{remark}
\newtheorem{remark}[theorem]{Remark}
\newtheorem{example}[theorem]{Example}
\newcommand{\gin}{\operatorname{gin}}
\newcommand{\inin}{\operatorname{in}}
\newcommand{\rank}{\operatorname{rank}}
\newcommand{\supp}{\operatorname{supp}}
\newcommand{\B}{\mathcal B}
\newcommand{\vct}{\mathbf v}
\newcommand{\Vct}{\mathbf V}
\newcommand{\Z}{\mathbb Z}
\title[Hilbert functions and uniqueness of generic initial ideals]{When Hilbert functions determine generic initial ideals}
\author{Van Duc Trung}
\address{Department of Mathematics, University of Education, Hue University,
34 Le Loi, Hue, Vietnam}
\email{vdtrung@hueuni.edu.vn}
\date{}
\begin{document}
\begin{abstract}
Over a field of characteristic zero, we characterize the Hilbert
functions that determine the degree reverse lexicographic generic
initial ideal among Artinian quotients of embedding dimension $n\ge2$ with the
$(n-2)$-strong Lefschetz property. The criterion is a finite arithmetic
test on successive positive differences of the Hilbert function.
It uses recursively defined integer vectors and requires no
construction of a monomial ideal. When uniqueness holds, the same
vectors recover the common generic initial ideal degree by degree.
The proof reduces uniqueness
successively to uniqueness of lower sets of specified cardinalities
in Borel orders. In four variables, the test simplifies to avoidance
of explicitly computed integer intervals. We recover the
quasi-symmetry criterion of Harima and Wachi and give infinite
families beyond that criterion in every number of variables at
least four.
\end{abstract}
\maketitle
\raggedbottom
\enlargethispage{4pt}

\section{Introduction}\label{sec:intro}

Let $R_n=K[x_1,\ldots,x_n]$ be standard graded over a field $K$ of
characteristic zero. We use degree reverse lexicographic order with
$x_1>\cdots>x_n$. For a homogeneous ideal $Q\subset R_n$, write
$\inin(Q)$ for its initial ideal and call the monomials outside
$\inin(Q)$ the \emph{standard monomials} of $Q$. The \emph{generic
initial ideal} of $Q$ is $\gin(Q)=\inin(gQ)$ for a general linear
change of coordinates $g$. The Gr\"obner degeneration of $gQ$ gives
a flat family over the affine line with fiber $R_n/\gin(Q)$ at zero
and fibers isomorphic to $R_n/Q$ at all nonzero scalars
\cite[Theorem~3.1.2 and Corollary~3.2.6]{HerzogHibi}.

Fix a class of homogeneous ideals in $R_n$ with Artinian quotients.
For each Hilbert function $h$ occurring among these quotients, let
$\mathcal{I}_h$ consist of all ideals $Q$ in the class for which $R_n/Q$
has Hilbert function $h$.
For $Q\in\mathcal{I}_h$ we have
\[
 h_q=\dim_K(R_n/Q)_q
    =\dim_K(R_n/\inin(Q))_q
    =\dim_K(R_n/\gin(Q))_q \qquad(q\ge0).
\]
Thus $h_q$ gives the number of degree-$q$ monomials outside $\gin(Q)$,
but the monomials themselves may depend on $Q$. If $\gin(Q)=J$ for every
$Q\in\mathcal{I}_h$, then each corresponding quotient $R_n/Q$ admits
a flat degeneration to $R_n/J$. It is therefore natural to ask for which Hilbert
functions $h$ all ideals in $\mathcal{I}_h$ have the same
\mbox{generic initial ideal}.

For an Artinian standard graded $K$-algebra $A$, a linear form
$\ell\in A_1$ is a \emph{strong Lefschetz element} if
$\ell^e:A_q\to A_{q+e}$ has rank $\min\{\dim_K A_q,\dim_K A_{q+e}\}$
for all $q\ge0$ and $e\ge1$. Its existence is the \emph{strong
Lefschetz property} (SLP), and $A/\ell A$ is the corresponding
Lefschetz quotient. For $k\ge1$, we say that $A$ has \emph{$k$-SLP}
if there are $\ell_1,\ldots,\ell_k\in A_1$ such that the image of
$\ell_i$ is strong Lefschetz on $A/(\ell_1,\ldots,\ell_{i-1})A$
for each $i$. The condition $0$-SLP is empty.

We consider all Artinian quotients
of embedding dimension $n\ge2$ with \mbox{$(n-2)$-SLP}. In the quotient
by the gin, the last variables then act successively as strong
Lefschetz elements (Lemma~\ref{lem:ginclass}). Maximal~rank therefore
prescribes how many standard monomials in each degree remain standard
after multiplication by a given power of the corresponding variable.

Harima and Wachi \cite[Theorem~1]{HarimaWachi} prove that for an
Artinian quotient $R_3/Q$ with SLP, its Hilbert function determines
$\gin(Q)$. They identify the common gin as the \emph{almost reverse
lexicographic} (or \emph{almost revlex}) ideal for that function.
A monomial ideal is almost revlex if, for every minimal generator
$u$, it contains each monomial $v>u$ with $\deg v=\deg u$. At most one such ideal has a given quotient Hilbert
function \cite[Remark~11(i)]{HarimaWachi}.

Their higher-dimensional result uses the positive differences of $h$.
We omit trailing zeros from finite sequences. For $g=(g_0,\ldots,g_e)$,
set $g_q=0$ outside $0\le q\le e$. Write $(a)_+=\max\{a,0\}$ and
define $(\Delta g)_q=(g_q-g_{q-1})_+$, $\Delta^0g=g$, and
$\Delta^{j+1}g=\Delta(\Delta^jg)$ for $j\ge0$. A sequence is
\emph{unimodal} if it is nondecreasing up to a maximum and nonincreasing
thereafter. Let $s$ be the first index at which $g$ attains its
maximum. We call $g$ \emph{quasi-symmetric} if it is unimodal and
$g_t\in\{g_0,\ldots,g_s\}$ for every $s<t\le e$.
For $n\ge4$, Harima and Wachi \cite[Theorem~2]{HarimaWachi} show
that under $n$-SLP the gin is the almost revlex ideal with quotient
Hilbert function $h$ whenever $\Delta^j h$ is quasi-symmetric for
every $0\le j\le n-4$.
Here $n$-SLP and $(n-2)$-SLP are equivalent because every Artinian
quotient in at most two variables has SLP in characteristic zero
\cite[Proposition~4.4]{HarimaMiglioreNagelWatanabe}. Even without
quasi-symmetry, an almost revlex ideal exists for every Hilbert
function in this class \cite[Proposition~24]{HarimaWachi}. Harima and
Wachi also give a four-variable example in which two distinct gins
occur for $h=(1,4,8,7)$ \cite[Example~21(i)]{HarimaWachiNagoya}.

We decide from $h$ alone whether $\gin(Q)$ equals the almost revlex
ideal for every Artinian quotient $R_n/Q$ with Hilbert function $h$
and $(n-2)$-SLP. In every embedding dimension $n\ge2$, our criterion consists of
finitely many coordinatewise comparisons of integer vectors computed
from $h$ and its positive differences.
For example, the Hilbert functions $(1,5,15,7)$ and $(1,5,15,8)$
have the same positive difference $(1,4,10)$, and neither is
quasi-symmetric. Among five-variable quotients with $3$-SLP, exactly
one gin occurs for the first Hilbert function and exactly two occur
for the second (Example~\ref{ex:five}).
The proof describes which multiples of standard monomials
by powers of the last variable can remain standard
(Proposition~\ref{prop:fiber}) and determines when there is only one
such choice.

To state our main theorem, we first make precise what it means for a
Hilbert function to determine the gin.
Fix $n\ge2$ and let $h=(1,n,h_2,\ldots,h_\delta)$ be a finite
sequence of positive integers. We call $h$ \emph{admissible in
$n$ variables} if there is a homogeneous ideal $Q\subset R_n$
such that $R_n/Q$ is Artinian with Hilbert function $h$ and
$(n-2)$-SLP. We say that an admissible $h$ has the
\emph{uniqueness property in $n$ variables} if all such ideals $Q$
have the same generic initial ideal.

For any finite sequence $h$ of positive integers with $h_0=1$ and
$h_1=n$, set
\begin{equation}\label{eq:iterates}
 h^{[0]}=h,\qquad h^{[j+1]}=\Delta h^{[j]}
 \quad(0\le j<n-2).
\end{equation}

Assume now that $h$ is admissible in $n$ variables. If $R_n/Q$ has
Hilbert function $h$, then $h_1=\dim_K(R_n/Q)_1=n-\dim_K Q_1$.
Thus $h_1=n$ gives $Q_1=0$, so the presentation is minimal and the
quotient has embedding dimension $n$.
In a presentation $R_r/I$ with $r>n$, one can eliminate the $r-n$
independent linear relations. In general coordinates the leading
monomials of $I_1$ are $x_1,\ldots,x_{r-n}$, the degree-one minimal
generators of $\gin(I)$.
Proposition~\ref{prop:admissibility} gives a numerical test for
admissibility when $n\ge3$.

Choose an Artinian quotient $A=R_n/Q$ with Hilbert function $h$ and
forms $\ell_1,\ldots,\ell_{n-2}\in A_1$ realizing $(n-2)$-SLP.
Passing to the quotient by a strong Lefschetz element replaces the
Hilbert function $g$ by $\Delta g$, since multiplication by that
element has maximal rank between consecutive degrees. Applying this at
each step shows that $A/(\ell_1,\ldots,\ell_j)A$ has Hilbert function
$h^{[j]}$ for $0\le j\le n-2$, with $A$ itself at $j=0$. In
particular, $h^{[j]}_0=1$ and $h^{[j]}_1=n-j$. We therefore use
$n-j$ variables for $h^{[j]}$, with two variables at the final step
$j=n-2$.

For a finite numerical sequence $g$, write
$\supp g=\{q\ge0:g_q\ne0\}$ and $\max g=\max_{q\ge0}g_q$. Set
\begin{equation}\label{eq:supports}
 \delta_j=\max\supp h^{[j]}\qquad(0\le j\le n-2).
\end{equation}
For $0\le j<n-2$, the Hilbert function $h^{[j+1]}$ has support
$\{0,\ldots,\delta_{j+1}\}$. Since $h^{[j+1]}=\Delta h^{[j]}$,
the sequence $h^{[j]}$ is strictly increasing through degree
$\delta_{j+1}$ and nonincreasing thereafter.
For each such $j$ and all $q,t$ with $0\le q\le\delta_{j+1}$ and
$t\ge q$, put
\begin{equation}\label{eq:counts}
 \rho_{j,q}(t)=\min\{h^{[j]}_q,h^{[j]}_t\}
             -\min\{h^{[j]}_{q-1},h^{[j]}_t\}.
\end{equation}
Thus $0\le \rho_{j,q}(t)\le h^{[j+1]}_q$, with equality on the
right when $q\le t\le\delta_{j+1}$.
Let $M\subset R_{n-j}$ be a monomial ideal whose quotient has Hilbert
function $h^{[j]}$ and strong Lefschetz element $x_{n-j}$.
There are $h^{[j+1]}_q$ degree-$q$ monomials $u\in R_{n-j-1}$
outside $M$. Exactly $\rho_{j,q}(t)$ of them satisfy
$u x_{n-j}^{t-q}\notin M$ (Lemma~\ref{lem:counts}).

To decide whether $\rho_{j,q}(t)$ determines which of the monomials
$u x_{n-j}^{t-q}$ are standard, we compare partial sums of exponents.
The recursion below computes these sums for the degree-$q$ standard
monomials of the almost revlex ideal with quotient Hilbert function
$h^{[j+1]}$ using only the sequences in \eqref{eq:iterates}.
For $1\le\nu\le\max h^{[n-2]}$, set
\begin{equation}\label{eq:vectorbase}
 \vct_{n-2}(\nu)=(\nu-1).
\end{equation}
Take the indices $0\le j\le n-3$ in decreasing order.
For $1\le\nu\le\max h^{[j]}$, let
\begin{equation}\label{eq:threshold}
 s_j(\nu)=\min\{r\ge0:h^{[j]}_r\ge\nu\},\qquad
 \lambda_j(\nu)=\nu-h^{[j]}_{s_j(\nu)-1},
\end{equation}
and define, by concatenation,
\begin{equation}\label{eq:vector}
 \vct_j(\nu)
   =\bigl(\vct_{j+1}(\lambda_j(\nu)),\ s_j(\nu)\bigr)
   \in\Z_{\ge0}^{n-j-1}.
\end{equation}
The recursion is well-defined because the minimality of $s_j(\nu)$ gives
\[
 1\le\lambda_j(\nu)
 \le h^{[j]}_{s_j(\nu)}-h^{[j]}_{s_j(\nu)-1}
 =h^{[j+1]}_{s_j(\nu)}.
\]
Maxima, minima, and inequalities between vectors are taken
coordinatewise.

\Needspace{10\baselineskip}
Our main result is the following numerical criterion.

\begin{introtheorem}\label{thm:main}
Let $h$ be admissible in $n$ variables. The following are equivalent.
\begin{enumerate}
\item $h$ has the uniqueness property in $n$ variables.
\item For all integers $j,q,t$ with $0\le j\le n-4$ and
$0\le q\le\delta_{j+1}<t\le\delta_j$, put $N=h^{[j+1]}_q$ and
$k=\rho_{j,q}(t)$. Whenever $0<k<N$, one has
\begin{equation}\label{eq:test}
 \max_{1\le\nu\le k}\vct_{j+1}(\nu)
 \ \le\
 \min_{k<\nu\le N}\vct_{j+1}(\nu).
\end{equation}
\end{enumerate}
\smallskip
For $n\le3$, condition~\textup{(2)} is empty, and uniqueness always
holds under the stated hypotheses.
\end{introtheorem}

When the inequalities hold, the vectors $\vct_0(\nu)$ give the
standard monomials of the common gin in every degree. This gin is the
unique almost revlex ideal with quotient Hilbert function $h$
(Proposition~\ref{prop:reconstruction}). If an inequality fails,
the proof constructs two distinct monomial ideals, each equal to
its own gin, whose quotients have Hilbert function $h$ and $(n-2)$-SLP.

For $n\ge4$ and $h$ quasi-symmetric, Corollary~\ref{cor:reduction}
reduces uniqueness for $h$ in $n$ variables to uniqueness for
$\Delta h$ in $n-1$ variables. Iterating gives the sufficient condition
of Harima and Wachi. For every $n\ge4$, Proposition~\ref{prop:families}
gives infinitely many admissible Hilbert functions $h$ with the
uniqueness property in $n$ variables for which $\Delta^j h$ is
symmetric for every $0\le j<n-4$, yet $\Delta^{n-4}h$ is not
quasi-symmetric.

\Needspace{10\baselineskip}
\section{Sections of strongly stable ideals and strong Lefschetz elements}\label{sec:lifts}

In this section $n\ge2$, and we use the natural inclusion
$R_{n-1}\subset R_n$. For a monomial ideal $M\subset R_n$, we call
$M\cap R_{n-1}$ the \emph{section} of $M$. We call $M$ \emph{strongly stable}
if $x_i u/x_j\in M$ whenever $u\in M$ is a monomial, $x_j\mid u$,
and $1\le i<j\le n$. The standard monomials of a strongly stable
ideal are closed under taking divisors and under replacing a variable
by one of higher index.
A subscript on $\inin$, as in $\inin_{R_n}(Q)$, specifies the
ambient ring. For the basic properties of generic initial ideals,
see \cite[Theorem~4.1.2 and Proposition~4.2.6]{HerzogHibi}.

\subsection{Generic initial ideals and successive Lefschetz elements}
For an Artinian quotient $R_n/Q$, Ahn, Cho, and Park
\cite[Theorem~3.7]{AhnChoPark} characterize SLP in terms of the
minimal generators of $\gin(Q)$. To keep track of successive
Lefschetz quotients, we use the last variables as Lefschetz
elements after passing to the gin.

\begin{lemma}\label{lem:compat}
For a homogeneous ideal $Q\subset R_n$ and $e\ge1$,
\[
 \inin(Q+(x_n^e))=\inin(Q)+(x_n^e).
\]
Consequently, multiplication by $x_n^e$ on $R_n/Q$ and on
$R_n/\inin(Q)$ has the same rank in every degree.
For $e=1$, write $\overline Q=(Q+(x_n))/(x_n)$ and identify
$R_n/(x_n)$ with $R_{n-1}$. Then
\begin{equation}\label{eq:section-compat}
 \inin_{R_{n-1}}(\overline Q)=\inin_{R_n}(Q)\cap R_{n-1}.
\end{equation}
\end{lemma}
\begin{proof}
The inclusion from right to left follows from
$Q,(x_n^e)\subseteq Q+(x_n^e)$. For the reverse inclusion,
take a nonzero homogeneous $f\in Q+(x_n^e)$ and write
$f=f_0+x_n^e a$ with $f_0\in Q$ and both summands homogeneous
of degree $\deg f$. If $x_n^e$ divides the leading monomial
of $f$, there is nothing to prove. Otherwise, that leading
monomial has $x_n$-exponent less than $e$. In degree reverse
lexicographic order it is larger than every monomial of
$x_n^e a$, so it is also the leading monomial of $f_0$.
It therefore belongs to $\inin(Q)$, proving the equality.
For either quotient $A$, the rank is
\[
 \rank(x_n^e:A_{t-e}\longrightarrow A_t)
 =\dim_K A_t-\dim_K(A/x_n^e A)_t \qquad(t\ge e),
\]
so the equality of initial ideals gives the rank assertion.

Setting $x_n=0$ identifies $\overline Q$ with
$(Q+(x_n))\cap R_{n-1}$. If a homogeneous $f\in Q+(x_n)$ has
leading monomial in $R_{n-1}$, its image in $\overline Q$ has the
same leading monomial. Conversely, each element of $\overline Q$,
viewed in $R_{n-1}$, belongs to $Q+(x_n)$. Hence
$\inin_{R_{n-1}}(\overline Q)=\inin_{R_n}(Q+(x_n))\cap R_{n-1}$.
The equality for $e=1$ now gives \eqref{eq:section-compat}, since
$x_n$ divides no monomial of $R_{n-1}$.
\end{proof}

For any finite sequence of positive integers
$h=(1,n,h_2,\ldots,h_\delta)$, let $\B_n(h)$ be the set of
Artinian strongly stable ideals $M\subset R_n$ such that $R_n/M$
has Hilbert function $h$ and
\[
 x_r\text{ is strong Lefschetz on }R_r/(M\cap R_r)
 \qquad(3\le r\le n).
\]
This set may be empty. For a monomial ideal $M$, one has
$R_r/(M\cap R_r)\simeq R_n/(M,x_{r+1},\ldots,x_n)$,
so the displayed condition specifies the successive Lefschetz
elements. For $n=2$ the range $3\le r\le n$ is empty.
The next lemma follows from \cite[Proposition~18]{HarimaWachi}
and the identity $\gin(M)=M$ for strongly stable monomial
ideals $M$. We include a proof for completeness.
For $h$ admissible in $n$ variables, uniqueness of the gin is
equivalent to $|\B_n(h)|=1$.

\begin{lemma}\label{lem:ginclass}
The elements of $\B_n(h)$ are precisely the generic initial
ideals of homogeneous ideals whose quotients have Hilbert
function $h$ and $(n-2)$-SLP. In particular, $h$ is admissible
in $n$ variables if and only if $\B_n(h)\ne\varnothing$.
\end{lemma}
\begin{proof}
Let $A=R_n/Q$ have the stated properties and put $M=\gin(Q)$.
The quotient $R_n/M$ is Artinian with Hilbert function $h$, and
$M$ is strongly stable in characteristic zero \cite[Proposition~4.2.6(a)]{HerzogHibi}.
For $n=2$ this gives $M\in\B_2(h)$. Assume $n\ge3$.

Fix bases of the graded pieces of $A$. The matrices of multiplication
by $\ell^e:A_q\to A_{q+e}$ have polynomial entries in the coefficients
of $\ell\in A_1$. Only the finitely many maps with $e\ge1$ and
$0\le q<q+e\le\delta$ need to be checked. For each of them, maximal
rank is equivalent to the nonvanishing of at least one maximal
minor. Thus the strong Lefschetz elements form an open subset of $A_1$.

On this open set, $A/\ell A$ has Hilbert function $\Delta h$.
Where chosen maximal minors of multiplication by $\ell$ are nonzero,
Gaussian elimination gives bases of the graded pieces of $A/\ell A$
whose coefficients depend rationally on $\ell$, with nonzero
denominators. In these bases, maximal rank for each power of the
next form is again an open condition. Repeating the argument shows
that the set $U\subset (A_1)^{n-2}$ of tuples realizing $(n-2)$-SLP
is open. It is nonempty by hypothesis.

Take $(\ell_1,\ldots,\ell_{n-2})\in U$. For $1\le i\le n-2$, set
$B=A/(\ell_1,\ldots,\ell_{i-1})A$, with $B=A$ for $i=1$. Since
$h_1=n$,
\[
 \dim_K B_1
 =n-\dim_K\langle\ell_1,\ldots,\ell_{i-1}\rangle_K
 \ge n-i+1>0.
\]
Multiplication by the image of $\ell_i$ from $B_0=K$ to $B_1$
has maximal rank, hence rank one. Thus $\ell_i$ lies outside the
span of the preceding forms. This proves linear independence and
gives the successive degree-one dimensions $n,n-1,\ldots,2$.

The equality $h_1=n$ also gives $Q_1=0$, so $A_1=(R_n)_1$.
For $g\in\mathrm{GL}_n(K)$, the induced isomorphism $A\to R_n/gQ$
takes $g^{-1}x_r$ to $x_r$. Thus the set of $g$ for which
$x_n,\ldots,x_3$ realize $(n-2)$-SLP on $R_n/gQ$ is the inverse
image of $U$ under the regular map $g\mapsto(g^{-1}x_n,\ldots,g^{-1}x_3)$.
This open set is nonempty because each tuple in $U$ extends to a basis
of $(R_n)_1$ and can therefore be sent to $x_n,\ldots,x_3$ in that order.
It meets the nonempty open set where $\inin(gQ)=M$ by irreducibility
of $\mathrm{GL}_n$. Since $\mathrm{GL}_n$ is open in affine space and
$K$ is infinite, the intersection contains a $K$-rational point.
Fix such a $g$.

For $3\le r\le n$, let $Q_r\subset R_r$ be the image of $gQ$ after
setting $x_{r+1},\ldots,x_n$ to zero. Then $Q_n=gQ$, and for $r>3$
setting $x_r=0$ sends $Q_r$ to $Q_{r-1}$. Starting from
$\inin_{R_n}(Q_n)=M$, we obtain by repeated application of
\eqref{eq:section-compat}
\[
 \inin_{R_r}(Q_r)=M\cap R_r\qquad(3\le r\le n).
\]
By the choice of $g$, $x_r$ is strong Lefschetz on $R_r/Q_r$.
The quotients $R_r/Q_r$ and $R_r/(M\cap R_r)$ have the same
Hilbert function. By Lemma~\ref{lem:compat}, multiplication by
$x_r^e$ has the same rank in each degree on both quotients for
every $e\ge1$. Hence $x_r$ is strong Lefschetz on $R_r/(M\cap R_r)$
for every $3\le r\le n$, so $M\in\B_n(h)$.

Conversely, if $M\in\B_n(h)$, then $M$ is strongly stable, hence
Borel-fixed, and therefore $\gin(M)=M$
\cite[Propositions~4.2.4(c) and~4.2.6(b)]{HerzogHibi}.
Taking $Q=M$ gives $\gin(Q)=M$, and $R_n/Q$ has Hilbert function
$h$ and $(n-2)$-SLP by the definition of $\B_n(h)$.
\end{proof}

\subsection{Monomial chains and ideals with a prescribed section}
For a monomial ideal $M$, let $E(M)$ be its set of standard
monomials and $E(M)_q$ its degree-$q$ part. Suppose that
$M\subset R_n$ is Artinian, with section $J=M\cap R_{n-1}$.
For each $u\in E(J)$ there is a unique $a\ge0$ such that
$u,ux_n,\ldots,ux_n^a$ are standard and $ux_n^{a+1}\in M$.
We call this sequence the $x_n$-chain starting at $u$.
We say that this chain \emph{reaches degree $t$} if it contains a
monomial of degree $t$, that is, if $\deg u\le t\le\deg u+a$.
Each standard monomial belongs to a unique such chain because
removing its highest power of $x_n$ recovers its starting monomial.
For this decomposition and the comparison of chain endpoints in
the last assertion of the next lemma, see
\cite[Definition and Remark~5, Definition~6(i), and Lemma~7]{HarimaWachi}.
The ranks of multiplication by powers of $x_n$ determine how many
chains start in degree $q$ and reach degree $t$.

\Needspace{12\baselineskip}
\begin{lemma}\label{lem:counts}
Let $M\subset R_n$ be an Artinian monomial ideal and set
$J=M\cap R_{n-1}$. Suppose that $x_n$ is strong Lefschetz on
$R_n/M$ and that its Hilbert function is $g$. Put $p=\Delta g$ and
$\tau=\max\supp p$. Then $R_{n-1}/J$ has Hilbert function $p$.
For $0\le q\le\tau$ and $t\ge q$, the number of chains starting in degree
$q$ and reaching degree $t$ is
\begin{equation}\label{eq:rho-local}
 \rho_q(t)=\min\{g_q,g_t\}-\min\{g_{q-1},g_t\}.
\end{equation}
For $h$ admissible in $n\ge3$ variables, taking $g=h$ in
\eqref{eq:rho-local} recovers the case $j=0$ of \eqref{eq:counts}.
For $0\le q<q'\le\tau$, if a chain starting in degree $q'$
reaches degree $t\ge q'$, then every chain starting in degree
$q$ reaches that degree.
\end{lemma}
\begin{proof}
Put $B=R_n/M$. Since $M$ is monomial, setting $x_n=0$ sends $M$
onto $J$. Thus $R_{n-1}/J\simeq B/x_n B$, whose degree-$q$ component
is the cokernel of multiplication by $x_n:B_{q-1}\to B_q$.
This map has maximal rank, so for $q\ge0$, with $B_{-1}=0$,
\[
 \dim_K\!\bigl((B/x_n B)_q\bigr)
 =g_q-\min\{g_{q-1},g_q\}=(g_q-g_{q-1})_+=p_q.
\]
In particular, $|E(J)_q|=p_q$.

For $0\le q\le\tau$ and $t\ge q$, the rank of $x_n^{t-q}:B_q\to B_t$
counts the chains starting in degree at most $q$ and reaching $t$.
Each such chain contributes one nonzero image, and these images are
distinct standard monomials of degree $t$. The rank from degree $q-1$
to degree $t$ counts those among them that start before $q$. Hence
the number of chains starting in degree $q$ and reaching $t$ is
\[
 \rho_q(t)=\rank(x_n^{t-q}:B_q\to B_t)
           -\rank(x_n^{t-q+1}:B_{q-1}\to B_t).
\]
By maximal rank, these ranks are $\min\{g_q,g_t\}$ and
$\min\{g_{q-1},g_t\}$, respectively, giving \eqref{eq:rho-local}.
When $t=q$ the first map is the identity, and when $q=0$ the
second has zero source.

The support of $p$ is $\{0,\ldots,\tau\}$ because a standard
graded algebra has no nonzero graded piece after a zero piece.
Thus $g$ is strictly increasing through $\tau$ and is
nonincreasing thereafter.

\Needspace{5\baselineskip}
Since $g_q=g_{q-1}+p_q$ for $0\le q\le\tau$,
equation~\eqref{eq:rho-local} becomes
\[
 \rho_q(t)=\min\{p_q,(g_t-g_{q-1})_+\}.
\]
It equals $p_q$ for $q\le t\le\tau$ and is nonincreasing in $t$
thereafter. Suppose a chain starting in degree $q'$, with
$q<q'\le\tau$, reaches $t\ge q'$. Then $\rho_{q'}(t)>0$, so
$g_t>g_{q'-1}$. Since $g$ is increasing through degree $\tau$,
$g_{q'-1}\ge g_q$. Hence $\rho_q(t)=g_q-g_{q-1}=p_q$.
Since $p_q$ is the total number of chains starting in degree $q$,
every such chain reaches degree $t$.
\end{proof}

We denote the \emph{Borel order} on monomials of a fixed degree in
$R_n$ by $\preceq$. Thus $u\preceq v$ if $u$ can be obtained from
$v$ by zero or more replacements of a factor $x_i$ by $x_j$ with $i<j$.
In terms of the exponents,
\begin{equation}\label{eq:borel}
 x^\alpha\preceq x^\beta
 \quad\Longleftrightarrow\quad
 \sum_{i=1}^a\alpha_i\le\sum_{i=1}^a\beta_i
 \quad(1\le a\le n-1).
\end{equation}
Each replacement weakly decreases every partial sum. Conversely,
the sum at $a$ counts factors with index at most $a$. Thus, in
nondecreasing order, the factor indices of $x^\beta$ are termwise
at most those of $x^\alpha$ when the inequalities hold. Replacing
the corresponding factors gives $x^\alpha$. In $R_{n-1}$ both sums
at $a=n-1$ equal the common degree, so only the first $n-2$ comparisons
are needed. Each replacement decreases the monomial in degree reverse
lexicographic order, so $u\preceq v$ implies $u\le v$ in that order.
Let $(P,\preceq)$ be a partially ordered set.
A subset $D\subseteq P$ is a \emph{lower set} if
\[
 v\in D,\quad u\in P,\quad u\preceq v
 \quad\Longrightarrow\quad u\in D.
\]
If $P$ is a set of monomials of the same degree with the Borel order,
the first $k$ monomials of $P$ in increasing degree reverse
lexicographic order form a lower set.

\Needspace{20\baselineskip}
With the section and the quotient Hilbert function fixed, we use lower
sets to record which $x_n$-chains reach each degree.

\begin{proposition}\label{prop:fiber}
Let $g$ be a finite sequence of positive integers with $g_0=1$.
Put $p=\Delta g$, $\tau=\max\supp p$, and $\delta=\max\supp g$.
Fix an Artinian strongly stable ideal $J\subset R_{n-1}$ such that
$R_{n-1}/J$ has Hilbert function $p$.

Then there exists an Artinian strongly stable ideal $M\subset R_n$
with section $J$ such that $R_n/M$ has Hilbert function $g$ and
strong Lefschetz element $x_n$. The set of all such ideals $M$ is
in bijection with the families of lower sets $D_{q,t}\subseteq E(J)_q$
such that, for $0\le q\le\tau$ and $q\le t\le\delta+1$,
\begin{equation}\label{eq:flags}
 E(J)_q=D_{q,q}\supseteq D_{q,q+1}\supseteq\cdots
 \supseteq D_{q,\delta+1}=\varnothing,\quad
 \text{with } |D_{q,t}|=\rho_q(t).
\end{equation}
The families for different values of $q$ can be chosen independently.
Here $\rho_q(t)$ is given by \eqref{eq:rho-local}.

Such an ideal $M$ is unique if and only if $E(J)_q$, ordered
by $\preceq$, has a unique lower set of cardinality $\rho_q(t)$
for all $0\le q\le\tau$ and $q\le t\le\delta$.
\end{proposition}
\Needspace{7\baselineskip}
\begin{proof}
Since $p$ is the Hilbert function of the standard graded algebra
$R_{n-1}/J$, its support is $\{0,\ldots,\tau\}$. As
$p_q=(g_q-g_{q-1})_+$, the sequence $g$ is strictly increasing
through degree $\tau$ and nonincreasing thereafter.

\smallskip
\noindent\emph{Step 1: Define the sets $D_{q,t}$ from $M$.}
For such an ideal $M$, let $D_{q,t}$ consist of all $u\in E(J)_q$
with $ux_n^{t-q}\notin M$.
If $u\preceq v$ in $E(J)_q$ and $v\in D_{q,t}$, strong stability
of $M$ gives $ux_n^{t-q}\notin M$, so $u\in D_{q,t}$.
Thus $D_{q,t}$ is a lower set. For $q\le t\le\delta$ and
$u\in D_{q,t+1}$, the divisor $ux_n^{t-q}$ of $ux_n^{t+1-q}$
is standard, so $u\in D_{q,t}$. Also, $D_{q,q}=E(J)_q$ since
$M\cap R_{n-1}=J$, and $D_{q,\delta+1}=\varnothing$ since $g_{\delta+1}=0$.

An element $u\in E(J)_q$ belongs to $D_{q,t}$ exactly when its
$x_n$-chain reaches degree $t$. Lemma~\ref{lem:counts} therefore
gives $|D_{q,t}|=\rho_q(t)$, as required in \eqref{eq:flags}.

\smallskip
\noindent\emph{Step 2: Construct $M$ from the sets $D_{q,t}$.}
For fixed $q$, the first $\rho_q(t)$ monomials of $E(J)_q$ in
increasing degree reverse lexicographic order form a lower set.
By \eqref{eq:rho-local}, the sizes $\rho_q(t)$ equal $p_q$ for
$q\le t\le\tau$ and are nonincreasing thereafter, with $\rho_q(\delta+1)=0$.
Thus the sets formed by the first $\rho_q(t)$ monomials satisfy
\eqref{eq:flags}.

For each $q$, now fix an arbitrary family of lower sets
$(D_{q,t})_{q\le t\le\delta+1}$ satisfying \eqref{eq:flags}, and set
\begin{equation}\label{eq:lift-proof}
 E=\bigcup_{q=0}^{\tau}\ \bigcup_{t=q}^{\delta}
       \{ux_n^{t-q}:u\in D_{q,t}\}.
\end{equation}
For $1\le q\le\tau$ and $q\le t\le\delta+1$, we have
\begin{equation}\label{eq:separation}
 D_{q,t}\ne\varnothing
 \quad\Longrightarrow\quad
 D_{q-1,t}=E(J)_{q-1}.
\end{equation}
Indeed, $D_{q,t}\ne\varnothing$ gives $\rho_q(t)>0$, so
$g_t>g_{q-1}$ by \eqref{eq:rho-local}. Applying the same formula
in degree $q-1$ gives
$|D_{q-1,t}|=g_{q-1}-g_{q-2}=p_{q-1}=|E(J)_{q-1}|$.

Write a monomial of $E$ as $ux_n^k$, where $q=\deg u$,
$t=q+k$, and $u\in D_{q,t}$.
If $k>0$, the inclusions give $u\in D_{q,t-1}$, so division by
$x_n$ stays in $E$. If $x_i\mid u$ with $i<n$, let $v=u/x_i$.
Then $v\in E(J)_{q-1}$ because every divisor of a standard
monomial of $J$ is standard. Equation~\eqref{eq:separation}
gives $v\in D_{q-1,t}$, hence also $v\in D_{q-1,t-1}$. Thus
$vx_n^{k+1}$ and $vx_n^k$ lie in $E$. These are obtained from
$ux_n^k$ by replacing $x_i$ with $x_n$ and by dividing by $x_i$,
respectively.
Finally, for $i<j<n$ and $x_i\mid u$, strong stability of
$J$ gives $x_j u/x_i\in E(J)_q$. This monomial is below $u$
in the Borel order, so the lower-set property gives
$(x_j u/x_i)x_n^k\in E$.
The only relation used between different starting degrees is
\eqref{eq:separation}. It follows from the prescribed cardinalities,
so the families for different $q$ can be chosen independently.

\Needspace{7\baselineskip}
Since $E$ is closed under taking divisors, every multiple of a
monomial outside $E$ is also outside $E$. Thus the monomials
outside $E$ span a monomial ideal $M$ with $E(M)=E$.
Let $w\in M$ be a monomial, with $x_j\mid w$ and $i<j$.
If $x_iw/x_j$ were in $E$, replacing $x_i$ by $x_j$ would give
$w\in E$. Hence $x_iw/x_j\in M$, so $M$ is strongly stable.
Since $E$ is finite,
$M$ is Artinian. In \eqref{eq:lift-proof}, $ux_n^{t-q}$ lies in
$R_{n-1}$ exactly when $t=q$. Hence
$E(M)\cap R_{n-1}=\bigcup_{q=0}^{\tau}D_{q,q}=E(J)$, and
$M\cap R_{n-1}=J$.

\smallskip
\noindent\emph{Step 3: Check the Hilbert function and SLP.}
Let $B=R_n/M$ and $0\le a\le b\le\delta$. Every degree-$a$
standard monomial has a unique expression $ux_n^{a-q}$ with
$0\le q\le\min\{a,\tau\}$ and $u\in D_{q,a}$. Its image under
$x_n^{b-a}$ is $ux_n^{b-q}$ if $u\in D_{q,b}$ and zero otherwise.
The nonzero images, indexed by $u\in D_{q,b}\subseteq D_{q,a}$,
are distinct elements of the standard monomial basis of $B_b$ and
hence linearly independent. Thus
\[
 \rank(x_n^{b-a}:B_a\longrightarrow B_b)
 =\sum_{q=0}^{\min\{a,\tau\}}\rho_q(b)
 =\min\{g_{\min\{a,\tau\}},g_b\}
 =\min\{g_a,g_b\}.
\]
Summing \eqref{eq:rho-local} cancels consecutive terms, with
$g_{-1}=0$. The last equality is immediate for $a\le\tau$.
For $a>\tau$, it follows from $g_b\le g_a\le g_\tau$.

Taking $a=b$ gives $\dim_K B_a=g_a$ since the map is the identity.
For $a<b\le\delta$ the computed rank is therefore maximal.
For $b>\delta$, \eqref{eq:lift-proof} gives $B_b=0$, so
$\dim_K B_b=g_b=0$ and every map into $B_b$ has maximal rank.
Thus $B$ has Hilbert function $g$ and $x_n$ is strong Lefschetz on $B$.

\smallskip
\noindent\emph{Step 4: Decide uniqueness.}
The two constructions are inverse: the sets $D_{q,t}$ specify
exactly which monomials $ux_n^{t-q}$ are outside $M$. If every
lower set of each required cardinality is unique, all $D_{q,t}$,
and hence $M$, are unique.

Conversely, fix $q,t$ for which two distinct lower sets $A,A'$
have cardinality $\rho_q(t)$. Order the elements of $A$ first,
followed by those of $E(J)_q\setminus A$, respecting
$\preceq$ within each part. Such an ordering is obtained
by repeatedly choosing a minimal remaining element. Since
$A$ is a lower set, no element outside $A$ lies below an
element of $A$; hence the entire ordering respects $\preceq$.
For every $q\le s\le\delta+1$, take the first $\rho_q(s)$ elements in this
ordering. These sets satisfy \eqref{eq:flags} and give
$D_{q,t}=A$. Repeating the construction with $A'$ gives
$D_{q,t}=A'$. Choose the same sets for every other value
of $q$, and denote the resulting ideals by $M$ and $M'$, respectively.
Since $A$ and $A'$ are distinct and have the same cardinality,
choose $u\in A\setminus A'$. By \eqref{eq:lift-proof},
$ux_n^{t-q}\in M'\setminus M$, so $M\ne M'$.
\end{proof}

For $n\ge3$ and $h$ admissible in $n$ variables, the section map
\begin{equation}\label{eq:section-map}
 \B_n(h)\longrightarrow\B_{n-1}(\Delta h),
 \qquad M\longmapsto M\cap R_{n-1},
\end{equation}
is well-defined and surjective. If $M\in\B_n(h)$,
then $R_{n-1}/(M\cap R_{n-1})$ has Hilbert function $\Delta h$
by Lemma~\ref{lem:counts}, and the Lefschetz conditions for
$3\le r\le n-1$ follow from $(M\cap R_{n-1})\cap R_r=M\cap R_r$.
Conversely, for $J\in\B_{n-1}(\Delta h)$,
Proposition~\ref{prop:fiber} constructs $M$ with
$M\cap R_{n-1}=J$ and strong Lefschetz element $x_n$.
For $3\le r<n$, we have $M\cap R_r=J\cap R_r$, so $M\in\B_n(h)$.
This gives the case $k=n-2$ of \cite[Proposition~21(i)]{HarimaWachi}.
Proposition~\ref{prop:fiber} also determines exactly when there is
only one such $M$ for a given $J$.

\Needspace{10\baselineskip}
\subsection{A numerical criterion for admissibility}
An $O$-sequence is the Hilbert function of a standard graded
$K$-algebra (see \cite[p.~4017]{HarimaWachi}).
For a finite sequence $h$ with $h_1=2$, Macaulay's theorem
\cite[Theorem~6.3.8]{HerzogHibi} shows that $h$ is an $O$-sequence
if and only if it is the Hilbert function of an Artinian quotient of
$R_2$. Since $0$-SLP imposes no further condition, this is precisely
what it means for $h$ to be admissible in two variables.
Harima and Wachi \cite[Proposition~23]{HarimaWachi} characterize the
$O$-sequences occurring as Hilbert functions of algebras with $k$-SLP.
For $k=n-2$, Proposition~\ref{prop:fiber} gives the following test.

\begin{proposition}\label{prop:admissibility}
Let $n\ge3$ and let $h=(1,n,h_2,\ldots,h_\delta)$ be a finite
sequence of positive integers. Then $h$ is admissible in $n$ variables
if and only if $h^{[n-2]}$ is an $O$-sequence, that is, the Hilbert
function of an Artinian quotient of $R_2$.
\end{proposition}
\begin{proof}
Each $h^{[j]}$ begins with $(1,n-j)$, since taking a positive
difference sends the first two entries $(1,r)$ to $(1,r-1)$ for $r\ge2$.
If $h$ is admissible, choose a quotient $A$ with Hilbert function
$h$ and forms $\ell_1,\ldots,\ell_{n-2}$ realizing $(n-2)$-SLP.
For a strong Lefschetz element $\ell$ on an algebra $B$ with Hilbert
function $g$, maximal rank gives
$\dim_K(B/\ell B)_q=g_q-\min\{g_{q-1},g_q\}=(\Delta g)_q$.
Thus $A/(\ell_1,\ldots,\ell_{n-2})A$ has Hilbert function $h^{[n-2]}$.

Conversely, suppose $h^{[n-2]}$ is an $O$-sequence. Its degree-one
entry is two, so the two-variable case and Lemma~\ref{lem:ginclass}
give an ideal in $\B_2(h^{[n-2]})$. For $j=n-3,\ldots,0$ in turn,
put $r=n-j$ and $g=h^{[j]}$. We have $\Delta g=h^{[j+1]}$, so the
ideal $J$ from the preceding step belongs to $\B_{r-1}(\Delta g)$.
The Hilbert function $\Delta g$ of the standard graded algebra
$R_{r-1}/J$ has support $\{0,\ldots,\tau\}$ for some $\tau$.
Thus $g$ increases strictly through degree $\tau$ and is nonincreasing
thereafter. After its trailing zeros are omitted, $g$ is a sequence of
positive integers. Proposition~\ref{prop:fiber} now gives
an Artinian strongly stable ideal $M\subset R_r$ with section $J$
such that $R_r/M$ has Hilbert function $g$ and strong Lefschetz
element $x_r$. For $3\le s<r$, the remaining Lefschetz conditions
follow from $M\cap R_s=J\cap R_s$. Thus $M\in\B_r(g)$.
At $j=0$ this gives $\B_n(h)\ne\varnothing$, so $h$ is admissible
in $n$ variables by Lemma~\ref{lem:ginclass}.
\end{proof}

\section{The arithmetic criterion and reconstruction}\label{sec:proof}

\subsection{Standard monomials and lower sets of prescribed cardinality}
Let $n\ge2$. For a monomial $u=x_1^{\alpha_1}\cdots x_n^{\alpha_n}$
of degree $q$, define its partial-sum vector by
\[
 \sigma_n(u)=
 \bigl(\alpha_1,\ \alpha_1+\alpha_2,\ \ldots,\
       \alpha_1+\cdots+\alpha_{n-1}\bigr)
 \in\Z_{\ge0}^{n-1}.
\]
For $n=2$ this is the one-entry vector $(\alpha_1)$.
We omit the full sum $\alpha_1+\cdots+\alpha_n=q$,
since it is fixed in degree $q$.
By \eqref{eq:borel}, two monomials $u,v$ of
degree $q$ satisfy $u\preceq v$ if and only if
$\sigma_n(u)\le\sigma_n(v)$ coordinatewise.

\Needspace{7\baselineskip}
Let $g$ be admissible in $n$ variables and fix
$1\le\nu\le\max g$. We define $\Vct_n(g;\nu)\in\Z_{\ge0}^{n-1}$
by recursion on $n$. For $n=2$, set $\Vct_2(g;\nu)=(\nu-1)$.
For $n\ge3$, put $s=\min\{r\ge0:g_r\ge\nu\}$ and define
\[
 \Vct_n(g;\nu)=\bigl(\Vct_{n-1}(\Delta g;\nu-g_{s-1}),\,s\bigr).
\]
Comparing this recursion with \eqref{eq:vectorbase}--\eqref{eq:vector} gives
\[
 \vct_j(\nu)=\Vct_{n-j}(h^{[j]};\nu)
 \qquad (0\le j\le n-2,\quad 1\le\nu\le\max h^{[j]}).
\]

\Needspace{16\baselineskip}
For each $g$ admissible in $n$ variables, the construction of Harima and Wachi
\cite[proofs of Propositions~21 and~24]{HarimaWachi} gives a strongly
stable ideal with quotient Hilbert function $g$. We compute the
partial-sum vectors of its standard monomials.

\begin{lemma}
\label{lem:canonical}
For every admissible sequence $g$ in $n\ge2$ variables there
is an ideal $C_n(g)\in\B_n(g)$, defined recursively below, with
the following property. If its degree-$q$ standard monomials
are listed increasingly as $u_1<\cdots<u_{g_q}$, then
\begin{equation}\label{eq:profile}
 \sigma_n(u_\nu)=\Vct_n(g;\nu)
 \qquad(1\le\nu\le g_q).
\end{equation}
\end{lemma}
\begin{proof}
\emph{The case $n=2$.}
Lemma~\ref{lem:ginclass} gives $\B_2(g)\ne\varnothing$.
Take $M\in\B_2(g)$. If $x_1^ix_2^{q-i}$ is standard for $M$ and
$i>0$, then replacing one factor $x_1$ by $x_2$ gives another
standard monomial. Since there are $g_q$ standard monomials in
degree $q$, they are exactly $x_1^ix_2^{q-i}$ with $0\le i<g_q$.
Thus all ideals in $\B_2(g)$ have the same standard monomials in
every degree and hence are equal. Denote this ideal by $C_2(g)$.
The monomials occur in increasing order of $i$, so
$u_\nu=x_1^{\nu-1}x_2^{q-\nu+1}$ and
$\sigma_2(u_\nu)=(\nu-1)=\Vct_2(g;\nu)$.

\Needspace{8\baselineskip}
\smallskip
\noindent\emph{Construction for $n\ge3$.}
Assume the lemma in $n-1$ variables. Put $p=\Delta g$ and
choose an Artinian quotient $A$
with Hilbert function $g$ and forms $\ell_1,\ldots,\ell_{n-2}$
realizing $(n-2)$-SLP. By maximal rank of multiplication by $\ell_1$
between consecutive degrees, $A/\ell_1A$ has Hilbert function $p$.
It is a quotient of $R_{n-1}$ since $p_1=n-1$, and the images
of the remaining forms give $(n-3)$-SLP on this quotient.
Thus $p$ is admissible in $n-1$ variables, and
induction gives $J=C_{n-1}(p)\in\B_{n-1}(p)$.

Set $\tau=\max\supp p$ and $\delta=\max\supp g$. For each
$0\le r\le\tau$, list the monomials of $E(J)_r$ in increasing
degree reverse lexicographic order. For $r\le t\le\delta+1$,
let $D_{r,t}$ consist of the first $\rho_r(t)$ monomials in this list.
Step~2 of the proof of Proposition~\ref{prop:fiber} shows that
these are lower sets satisfying \eqref{eq:flags}. By the bijection
in that proposition, exactly one ideal with section $J$ gives this
family. Denote it by $C_n(g)$.
The quotient $R_n/C_n(g)$ has Hilbert function $g$ and strong
Lefschetz element $x_n$. For $3\le i<n$ we have
$C_n(g)\cap R_i=J\cap R_i$, and $J\in\B_{n-1}(p)$ gives the
remaining Lefschetz conditions. Hence $C_n(g)\in\B_n(g)$.

\smallskip
\noindent\emph{Proof of \eqref{eq:profile}.}
Fix $q\ge0$ and $1\le\nu\le g_q$. By \eqref{eq:lift-proof},
the degree-$q$ standard monomials are $w x_n^{q-r}$ with
$0\le r\le\min\{q,\tau\}$ and $w\in D_{r,q}$. A smaller $r$
gives a larger exponent of $x_n$, so the monomials with smaller
$r$ come first. For fixed $r$, their order is that of $w$ in $E(J)_r$.
By \eqref{eq:rho-local}, the number with $\deg w\le r$ is
\[
 \sum_{a=0}^r\rho_a(q)
 =\sum_{a=0}^r\bigl(\min\{g_a,g_q\}-\min\{g_{a-1},g_q\}\bigr)
 =\min\{g_r,g_q\},
\]
since $g_{-1}=0$. Since $\nu\le g_q$, we have $\min\{g_r,g_q\}\ge\nu$
if and only if $g_r\ge\nu$. Put $s=\min\{r\ge0:g_r\ge\nu\}$.
Since $g_q\ge\nu$ and $g_\tau=\max g$, we have $s\le\min\{q,\tau\}$.
Thus $s$ is the least $r$ for which $\sum_{a=0}^{r}\rho_a(q)\ge\nu$.
Since monomials with smaller $\deg w$ come first, the factor $w$ of
$u_\nu$ has degree $s$. For $s>0$, the number of monomials with
$\deg w<s$ is $\min\{g_{s-1},g_q\}=g_{s-1}$, since $g_{s-1}<\nu\le g_q$.
For $s=0$, there are no such monomials and $g_{-1}=0$.
Among the monomials with $\deg w=s$, the position of $u_\nu$
is therefore $\lambda=\nu-g_{s-1}$, with
\[
 1\le\lambda\le\min\{g_s,g_q\}-g_{s-1}
 =\rho_s(q)\le p_s.
\]
Let $w_\lambda$ be the $\lambda$th monomial of $E(J)_s$.
Since $D_{s,q}$ consists of the first $\rho_s(q)$ monomials of
$E(J)_s$, its $\lambda$th monomial is also $w_\lambda$.
Thus $u_\nu=w_\lambda x_n^{q-s}$. The sum of its exponents
on $x_1,\ldots,x_{n-1}$ is $\deg w_\lambda=s$. By induction
and the definition of $\Vct_n$,
\[
 \sigma_n(u_\nu)
 =\bigl(\sigma_{n-1}(w_\lambda),s\bigr)
 =\bigl(\Vct_{n-1}(p;\lambda),s\bigr)
 =\Vct_n(g;\nu). \qedhere
\]
\end{proof}

\Needspace{14\baselineskip}
With the quotient Hilbert function and the section $J$ fixed,
Proposition~\ref{prop:fiber} reduces uniqueness to deciding whether
$E(J)_q$ has exactly one lower set of each required size $\rho_q(t)$.

\begin{lemma}\label{lem:poset}
Let $(P,\preceq)$ be a finite partially ordered set.
List the elements of $P$ as $u_1,\ldots,u_N$ so that
$u_a\preceq u_b$ implies $a\le b$.
For $0<k<N$, it has a unique lower set of cardinality $k$
if and only if
\begin{equation}\label{eq:cross}
 u_a\preceq u_b\qquad(1\le a\le k<b\le N).
\end{equation}
When unique, the lower set is $\{u_1,\ldots,u_k\}$.
\end{lemma}
\begin{proof}
Put $A=\{u_1,\ldots,u_k\}$. If $u_b\in A$ and $u_a\preceq u_b$,
then $a\le b\le k$, so $u_a\in A$. Thus $A$ is a lower set. If \eqref{eq:cross} holds and a
different lower set $D$ has size $k$, choose
$u\in A\setminus D$ and $v\in D\setminus A$. Then $u\preceq v$
contradicts the lower-set property of $D$.

If \eqref{eq:cross} fails, choose $u\in A$ and $v\notin A$
with $u\npreceq v$. Since $P$ is finite, choose a maximal element
$u^\ast$ of $A$ with $u\preceq u^\ast$ and a minimal element
$v^\ast$ of $P\setminus A$ with $v^\ast\preceq v$. We have $u^\ast\npreceq v^\ast$,
since otherwise $u\preceq u^\ast\preceq v^\ast\preceq v$.
Since $u^\ast$ is maximal in $A$, the set $A\setminus\{u^\ast\}$
is a lower set. If $w\preceq v^\ast$ and $w\ne v^\ast$, then
$w\in A$ by minimality of $v^\ast$ in $P\setminus A$. Moreover,
$w\ne u^\ast$ since $u^\ast\npreceq v^\ast$. Thus every such $w$
belongs to $A\setminus\{u^\ast\}$. Hence
$(A\setminus\{u^\ast\})\cup\{v^\ast\}$ is a lower set of size $k$
distinct from $A$.
The only lower sets of sizes $0$ and $N$ are $\varnothing$ and $P$,
respectively.
\end{proof}

\Needspace{12\baselineskip}
\subsection{Proof of Theorem~\ref{thm:main}}
\begin{proof}
By Lemma~\ref{lem:ginclass}, the required uniqueness is
equivalent to $|\B_n(h)|=1$. We argue by induction on $n$.
For $n=2$, the proof of Lemma~\ref{lem:canonical} gives
$\B_2(h)=\{C_2(h)\}$.

\Needspace{9\baselineskip}
Let $n\ge3$ and assume the theorem in $n-1$ variables.
Put $p=\Delta h$, $\tau=\max\supp p$, and $\delta=\max\supp h$.
As in the proof of Lemma~\ref{lem:canonical}, $p$ is admissible in
$n-1$ variables. The section map \eqref{eq:section-map} sends
$\B_n(h)$ onto $\B_{n-1}(p)$, which contains $J=C_{n-1}(p)$.
If $\B_{n-1}(p)=\{J\}$, then every $M\in\B_n(h)$ has section $J$.

We now determine when $J$ is the section of exactly one ideal in
$\B_n(h)$.
For $0\le q\le\tau$, list the monomials of $E(J)_q$ in increasing
degree reverse lexicographic order as $u_1<\cdots<u_{p_q}$.
By Lemma~\ref{lem:canonical}, $\sigma_{n-1}(u_\nu)=\Vct_{n-1}(p;\nu)$
for $1\le\nu\le p_q$.
Proposition~\ref{prop:fiber} and Lemma~\ref{lem:poset}, together with
the Borel comparison \eqref{eq:borel}, show that exactly one
$M\in\B_n(h)$ satisfies $M\cap R_{n-1}=J$ if and only if,
for every $0\le q\le\tau$ and $q\le t\le\delta$ with
$0<\rho_q(t)<p_q$, setting $k=\rho_q(t)$ gives
\begin{equation}\label{eq:one-step-test}
 \max_{1\le\nu\le k} \Vct_{n-1}(p;\nu)
 \le
 \min_{k<\nu\le p_q} \Vct_{n-1}(p;\nu).
\end{equation}
By \eqref{eq:rho-local}, $\rho_q(t)=p_q$ for $q\le t\le\tau$, so the
inequality $\rho_q(t)<p_q$ forces $t>\tau$.

\Needspace{8\baselineskip}
For $n=3$, $\Vct_2(p;\nu)=(\nu-1)$, so every comparison in
\eqref{eq:one-step-test} holds. Since $\B_2(p)=\{J\}$, it follows
that $|\B_3(h)|=1$.

Now let $n\ge4$ and assume \textup{(2)}. Since
$\Delta^i p=h^{[i+1]}$, the comparisons for $1\le j\le n-4$
are exactly those for $p$ in $n-1$ variables. This range is empty
when $n=4$. By the induction hypothesis, $\B_{n-1}(p)=\{J\}$.
For $j=0$, we have $\delta_1=\tau$, $\delta_0=\delta$,
$\rho_{0,q}(t)=\rho_q(t)$, and $\vct_1(\nu)=\Vct_{n-1}(p;\nu)$.
Thus the comparisons \eqref{eq:test} for $j=0$ give
\eqref{eq:one-step-test}, so exactly one ideal in $\B_n(h)$ has
section $J$. Since every ideal in $\B_n(h)$ has that section,
$|\B_n(h)|=1$. This proves $\textup{(2)}\Rightarrow\textup{(1)}$.

Conversely, suppose that \textup{(2)} fails for a triple $(j,q,t)$.
For the Hilbert function $h^{[j]}$ in $n-j$ variables, take
$J=C_{n-j-1}(h^{[j+1]})\subset R_{n-j-1}$ as the section.
By Lemma~\ref{lem:canonical}, the $\nu$th monomial of $E(J)_q$ in
increasing order has partial-sum vector
$\Vct_{n-j-1}(h^{[j+1]};\nu)=\vct_{j+1}(\nu)$.
The failure of \eqref{eq:test}, together with \eqref{eq:borel} and
Lemma~\ref{lem:poset}, therefore gives two distinct lower sets of
size $\rho_{j,q}(t)$ in $E(J)_q$. By \eqref{eq:counts}, this is the
size prescribed by \eqref{eq:rho-local} for $h^{[j]}$.
Step~4 of the proof of Proposition~\ref{prop:fiber} extends each
choice to a family satisfying \eqref{eq:flags}, giving distinct
$M_1,M_2\in\B_{n-j}(h^{[j]})$ with section $J$.
If $j=0$, put $\widetilde M_i=M_i$. If $j>0$, surjectivity of
\eqref{eq:section-map} gives, for each $i$, an ideal in
$\B_{n-j+1}(h^{[j-1]})$ whose intersection with $R_{n-j}$ is $M_i$.
Repeating this step one variable at a time gives
$\widetilde M_i\in\B_n(h)$ with
\[
 \widetilde M_i\cap R_{n-j}=M_i\qquad(i=1,2).
\]
Since $M_1\ne M_2$, these intersections imply
$\widetilde M_1\ne\widetilde M_2$. Both quotients have Hilbert
function $h$ and $(n-2)$-SLP, and each ideal is its own gin by
Lemma~\ref{lem:ginclass}. Thus \textup{(1)} fails whenever
\textup{(2)} fails. This proves the equivalence.
\end{proof}

\begin{remark}[Implementing the test]\label{rem:algorithm}
Fix $0\le j\le n-4$ and put $g=h^{[j]}$. Consider each distinct value
$y=g_t$ with $\delta_{j+1}<t\le\delta_j$ once. For $0\le r\le\delta_{j+1}$,
\eqref{eq:counts} gives $0<\rho_{j,r}(t)<(\Delta g)_r$ exactly when
$g_{r-1}<y<g_r$. Since $g$ increases strictly through $\delta_{j+1}$,
the only possible $r$ is $q=s_j(y)$. If $g_q=y$, no comparison is
needed. Otherwise check \eqref{eq:test} with $k=y-g_{q-1}$ and
$N=g_q-g_{q-1}$. For fixed $j,N$, the successive maxima from the left
and minima from the right need only be computed once.
\end{remark}

\Needspace{14\baselineskip}
\subsection{Reconstruction of the almost revlex ideal}
For $h$ admissible in $n$ variables, the existence of an almost
revlex ideal with quotient Hilbert function $h$ follows from
\cite[Proposition~24]{HarimaWachi}. Pardue \cite[Theorem~4]{Pardue}
characterizes the Hilbert functions admitting such an ideal.
For complete intersection Hilbert functions, Bertone and Cioffi
\cite[Theorem~4.1]{BertoneCioffi} give an explicit construction.
Here we identify $C_n(h)$ with the almost revlex ideal and recover
its standard monomials from the vectors $\vct_0(\nu)$.

\begin{proposition}
\label{prop:reconstruction}
Let $h$ be admissible in $n$ variables. Then $C_n(h)$ is the
unique almost revlex ideal with quotient Hilbert function $h$.
For $1\le\nu\le\max h$, write
\[
 \vct_0(\nu)=\bigl(b_1(\nu),\ldots,b_{n-1}(\nu)\bigr),
 \qquad b_0(\nu)=0.
\]
For each $q\ge0$, the degree-$q$ standard monomials of $C_n(h)$ are
exactly
\begin{equation}\label{eq:reconstruction}
 \left\{
 \left(\prod_{i=1}^{n-1}x_i^{\,b_i(\nu)-b_{i-1}(\nu)}\right)
 x_n^{\,q-b_{n-1}(\nu)}
 :1\le\nu\le h_q
 \right\}.
\end{equation}
If $h$ satisfies the equivalent conditions of
Theorem~\ref{thm:main}, this ideal is the common gin of all
homogeneous ideals whose quotients have Hilbert function $h$
and $(n-2)$-SLP.
\end{proposition}
\begin{proof}
\emph{Almost revlex identification.}
We prove by induction on $n$ that $C_n(g)$ is almost revlex
for every $g$ admissible in $n$ variables. For $n=2$,
Lemma~\ref{lem:canonical} gives the degree-$q$ standard monomials
$x_1^i x_2^{q-i}$ with $0\le i<g_q$. Since the degree-$q$ monomials
increase with $i$, all those larger than a degree-$q$ monomial of
$C_2(g)$ also belong to $C_2(g)$.
Let $n\ge3$ and put $J=C_{n-1}(\Delta g)$, which is almost
revlex by induction.
Let $u x_n^a$ be a minimal generator, where $x_n\nmid u$,
and let $v x_n^b>u x_n^a$ be a monomial of the same total
degree $d$, with $x_n\nmid v$. By degree reverse lexicographic
comparison, $b\le a$, and if $b=a$, then $v>u$.
If $a=0$, then $b=0$ and
$u$ is a minimal generator of $J$; induction gives
$v\in J\subset C_n(g)$. Suppose $a>0$. Minimality implies
$ux_n^{a-1}\notin C_n(g)$, and hence $u\notin J$.
If $v\in J$, then $vx_n^b\in C_n(g)$. It remains to
consider $v\notin J$. Suppose, for a contradiction, that
$vx_n^b$ is standard.

If $b=a$, then $v\in D_{d-a,d}$. By the construction in the proof of
Lemma~\ref{lem:canonical}, $D_{d-a,d}$ consists of the first
$\rho_{d-a}(d)$ monomials of $E(J)_{d-a}$ in increasing degree reverse
lexicographic order. Since $u<v$ and $v\in D_{d-a,d}$, we have
$u\in D_{d-a,d}$. Thus $ux_n^a$ is standard,
a contradiction. If $b<a$, then
$\deg u=d-a<d-b=\deg v$. The chain starting at $v$
reaches degree $d$. By Lemma~\ref{lem:counts}, every chain
starting in degree $\deg u$ reaches $d$, including the one
starting at $u$. This again makes $ux_n^a$ standard.
Thus every larger monomial of degree $d$ belongs to $C_n(g)$,
which proves that $C_n(g)$ is almost revlex.
By \cite[Remark~11(i)]{HarimaWachi}, $C_n(g)$ is the unique
almost revlex ideal with quotient Hilbert function $g$.

\smallskip
\noindent\emph{Reconstruction and the common gin.}
By Lemma~\ref{lem:canonical}, the $\nu$th standard monomial
$u_\nu=x_1^{\alpha_1}\cdots x_n^{\alpha_n}$ in the increasing
degree-$q$ list satisfies
$\alpha_1+\cdots+\alpha_i=b_i(\nu)$ for $1\le i<n$.
Taking consecutive differences gives
$\alpha_i=b_i(\nu)-b_{i-1}(\nu)$, and the total degree gives
$\alpha_n=q-b_{n-1}(\nu)$. This proves
\eqref{eq:reconstruction}. In particular,
$0=b_0(\nu)\le b_1(\nu)\le\cdots\le b_{n-1}(\nu)\le q$
for $\nu\le h_q$, so all displayed exponents are nonnegative.
When $h_q=0$, the displayed set is empty. Its complement in
each degree specifies the monomials of $C_n(h)$ and therefore
determines the entire ideal.

Since $C_n(h)\in\B_n(h)$, Theorem~\ref{thm:main} and
Lemma~\ref{lem:ginclass} give the common gin assertion
whenever the uniqueness criterion holds.
\end{proof}

The formula describes the ideal $C_n(h)$ for every
admissible $h$, including those for which other gins exist.
Uniqueness among almost revlex ideals is therefore distinct
from the uniqueness tested by Theorem~\ref{thm:main}.

\Needspace{12\baselineskip}
\section{Consequences of Theorem A}\label{sec:consequences}

\subsection{Reduction under quasi-symmetry}
When $h$ is quasi-symmetric, Theorem~\ref{thm:main} requires no
comparison at $j=0$.

\begin{corollary}\label{cor:reduction}
Let $n\ge4$ and let $h$ be quasi-symmetric and admissible in $n$ variables.
Then $h$ has the uniqueness property in $n$ variables if and only if
$\Delta h$ has it in $n-1$ variables. Consequently, if $h^{[j]}$ is quasi-symmetric for
every $0\le j\le n-4$, then $h$ has the uniqueness property in $n$ variables.
\end{corollary}
\begin{proof}
For $\delta_1<t\le\delta_0$, quasi-symmetry gives
$h_t=h_r$ for some $0\le r\le\delta_1$. Fix $0\le q\le\delta_1$.
Since $h_0,\ldots,h_{\delta_1}$ are strictly increasing, $r<q$
gives $h_t\le h_{q-1}$, while $r\ge q$ gives $h_t\ge h_q$.
Thus \eqref{eq:counts} gives
\[
 \rho_{0,q}(t)=
 \begin{cases}
  0 & \text{if } r<q,\\
  (\Delta h)_q & \text{if } r\ge q.
 \end{cases}
\]
Consequently, the conditions for $j=0$ in
Theorem~\ref{thm:main} impose no restriction. The conditions
for $j\ge1$ are exactly those for $\Delta h$, because
$\Delta^j h=\Delta^{j-1}(\Delta h)$. Repeating this reduction
ends in three variables, where uniqueness holds by
Theorem~\ref{thm:main}.
\end{proof}

Under the hypotheses of Corollary~\ref{cor:reduction}, the section
map \eqref{eq:section-map} is bijective
\cite[Proposition~21(ii)]{HarimaWachi}. This gives the equivalence
in the corollary. Iterating the corollary gives uniqueness, and
Proposition~\ref{prop:reconstruction} identifies the common gin as the
almost revlex ideal. This recovers \cite[Theorem~2]{HarimaWachi}, since
$(n-2)$-SLP and $n$-SLP are equivalent here
\cite[Remark~16]{HarimaWachi}. In three variables, Cho and Park
\cite[Proposition~2.7]{ChoPark} characterize SLP by the almost revlex
property of the gin. Cimpoea\c{s} \cite[Theorem~1.1]{Cimpoeas} shows
that the gin of a three-variable complete intersection with SLP is
almost revlex.

\begin{remark}[Four-variable complete intersections]\label{rem:ci}
Let $R_4/I$ be an Artinian complete intersection of embedding
dimension four with $2$-SLP. Let $d_1,\ldots,d_4$ be the degrees
of a homogeneous regular sequence generating $I$. The Hilbert
series of $R_4/I$ is $\prod_{i=1}^4(1+z+\cdots+z^{d_i-1})$,
so its Hilbert function is symmetric. SLP gives unimodality, hence
the Hilbert function is quasi-symmetric.
Corollary~\ref{cor:reduction} and the three-variable case of
Theorem~\ref{thm:main} show that the Hilbert function of $R_4/I$
has the uniqueness property in four variables.
Proposition~\ref{prop:reconstruction} then identifies $\gin(I)$ with the unique almost revlex ideal having that
quotient Hilbert function, as in \cite[Corollary~27]{HarimaWachi}.
\end{remark}

\Needspace{16\baselineskip}
\subsection{An interval criterion in four variables}
Four is the first embedding dimension in which the comparisons in
Theorem~\ref{thm:main} can fail. The values of $h$ up to its first
maximum determine the intervals below. Uniqueness holds exactly when
none of the subsequent values of $h$ lies in any of these intervals.

\begin{proposition}\label{prop:four}
Let $h=(1,4,h_2,\ldots,h_\delta)$ be admissible in four variables.
Put $p=\Delta h$ and $\tau=\max\supp p$.
For each $0\le q\le\tau$ with $p_q\ge4$, define
\begin{align}
 r_q&=\min\{r\ge0:p_r\ge p_q\},\notag\\
 a_q&=\max_{0\le r<r_q}(p_r-p_{r-1}),\label{eq:interval-data}\\
 b_q&=\min\{p_q-1,\ p_{r_q-1}+a_q-2\}.\notag
\end{align}
Then $h$ has the uniqueness property in four variables if and only if
\begin{equation}\label{eq:interval-test}
 h_t\notin[\,h_{q-1}+3,\ h_{q-1}+b_q\,]\cap\Z
 \quad\text{for every }q\text{ with }p_q\ge4
 \text{ and }\tau<t\le\delta.
\end{equation}
In particular, uniqueness holds whenever $\max p\le3$.
\end{proposition}
\begin{proof}
In Theorem~\ref{thm:main}, only $j=0$ needs to be checked.
Fix $0\le q\le\tau$ and put $N=p_q$. Write
$\vct_1(\nu)=(\eta_\nu,s_1(\nu))$. The second coordinate
$s_1(\nu)=\min\{r:p_r\ge\nu\}$ is nondecreasing, so its
inequality in \eqref{eq:test} always holds. The first
coordinate is
$\eta_\nu=\nu-p_{s_1(\nu)-1}-1$ by \eqref{eq:vectorbase}--\eqref{eq:vector}.
Since $\Delta p=h^{[2]}$ has support $\{0,\ldots,\delta_2\}$,
$p$ increases strictly through $\delta_2$ and is nonincreasing thereafter.
Put $s=s_1(N)$. As $N=p_q\le p_{\delta_2}$, we have $s\le\delta_2$.
For $1\le\nu\le N$ and $0\le r\le s$, it follows that
\[
 s_1(\nu)=r\quad\Longleftrightarrow\quad
 p_{r-1}<\nu\le\min\{p_r,N\}.
\]
For each $r$, subtracting $p_{r-1}+1$ from these consecutive indices
gives the first coordinates $(0,1,\ldots,c_r-1)$, where
$c_r=\min\{p_r,N\}-p_{r-1}$. Taking these lists in order from $r=0$
to $r=s$ gives $(\eta_1,\ldots,\eta_N)$.
If $N\le3$, then $p_0=1$ and $p_1=3$ show that this list
consists of the first $N$ entries of $(0,0,1)$.
Thus \eqref{eq:test} holds for every $0<k<N$.

Suppose $N\ge4$. Then $s=r_q$ by \eqref{eq:interval-data}.
Put $P=p_{s-1}$ and $a=a_q$.
The lists are indexed by $0,\ldots,s$.
For $r<s$, their lengths are $p_r-p_{r-1}$;
the last list has length $N-P$. The first two lists are
$(0)$ and $(0,1)$; in particular, $s\ge2$, $P\ge3$,
and $a\ge2$.

For $0<k<N$, it remains to determine when
\[
 \max_{1\le\nu\le k}\eta_\nu
 \le\min_{k<\nu\le N}\eta_\nu.
\]
For $k=1,2$, the left side is zero and all entries on the
right are nonnegative. Suppose $k>2$. Then $\eta_3=1$
occurs among the first $k$ entries. If $k<P$, the initial
zero of a later list occurs among the remaining entries,
so the inequality fails. We may therefore assume $k\ge P$
and write $k=P+u$, where $0\le u<N-P$.
The remaining entries are $u,u+1,\ldots,N-P-1$.
The largest entry in the preceding complete lists is $a-1$;
the entries already used in the last list are
$0,\ldots,u-1$ when $u>0$. The condition is therefore
\[
 \max\{a-1,u-1\}\le u
 \quad\Longleftrightarrow\quad k\ge P+a-1.
\]
Thus the values of $k$ for which the inequality fails are exactly
$3\le k\le\min\{N-1,P+a-2\}=b_q$.
Finally,
\[
 \rho_{0,q}(t)=\min\{p_q,(h_t-h_{q-1})_+\}.
\]
Since $3\le b_q\le p_q-1$, one has
$3\le\rho_{0,q}(t)\le b_q$ if and only if
$h_{q-1}+3\le h_t\le h_{q-1}+b_q$.
This proves \eqref{eq:interval-test}.
\end{proof}

Since $b_q<p_q=h_q-h_{q-1}$ for $p_q\ge4$, the corresponding
interval in \eqref{eq:interval-test} is contained in $(h_{q-1},h_q)$.
These intervals are pairwise disjoint because $h$ is strictly increasing
through degree $\tau$.

\Needspace{8\baselineskip}
\begin{example}\label{ex:four}
For integers $\ell\ge1$ and $11\le c\le19$, set
\[
 h^{(\ell,c)}=(1,4,10,20,
                  \underbrace{c,\ldots,c}_{\ell\text{ terms}}).
\]
Its positive difference is $p=(1,3,6,10)$, realized by
$K[x_1,x_2,x_3]/(x_1,x_2,x_3)^4$.
Multiplication by any power of $x_3$ is injective whenever
the target degree is at most three, and has zero target
otherwise, so this three-variable quotient has SLP.
Proposition~\ref{prop:fiber} therefore shows that every $h^{(\ell,c)}$
is admissible in four variables.

The only indices with $p_q\ge4$ are $q=2,3$. Their data are
\[
 (r_2,a_2,b_2)=(2,2,3),\qquad
 (r_3,a_3,b_3)=(3,3,7).
\]
Thus the forbidden intervals in Proposition~\ref{prop:four} are
$[7,7]$ and $[13,17]$. Since $11\le c\le19$, the criterion gives
\begin{equation}\label{eq:example-classification}
 h^{(\ell,c)}\text{ has the uniqueness property in four variables}
 \Longleftrightarrow c\in\{11,12,18,19\}.
\end{equation}
In particular, $h^{(\ell)}:=h^{(\ell,18)}$ has the
uniqueness property for every $\ell\ge1$, although it is not
quasi-symmetric: $18$ does not occur on $(1,4,10,20)$.

For $g=h^{(1,15)}=(1,4,10,20,15)$, every $M\in\B_4(g)$ has section
$J=(x_1,x_2,x_3)^4$, since $p$ counts all monomials of $R_3$ through
degree three and is zero thereafter. Likewise, $g$ forces all monomials
of $R_4$ of degree at most three to be standard and all of degree at
least five to belong to $M$. In degree four,
\eqref{eq:rho-local} gives $\rho_q(4)=p_q$ for $0\le q\le2$ and
$\rho_3(4)=15-10=5$. Thus for $0\le q\le2$ and $u\in E(J)_q$,
the monomial $ux_4^{4-q}$ must be standard. For $q=3$,
choose one of the sets
\[
 \begin{aligned}
 D&=\{x_3^3,x_2x_3^2,x_1x_3^2,x_2^2x_3,x_1x_2x_3\},\\
 D'&=(D\setminus\{x_1x_2x_3\})\cup\{x_2^3\}.
 \end{aligned}
\]
Indeed, by \eqref{eq:borel}, $D$ consists of the degree-three monomials
$u\preceq x_1x_2x_3$, and $D'$ consists of those with
$u\preceq x_1x_3^2$ or $u\preceq x_2^3$. Thus both are lower
sets of size five. To see that they are the only ones, let
$L\subseteq E(J)_3$ be a lower set of size five. If $x_1x_3^2\notin L$,
then \eqref{eq:borel} allows only the four degree-three monomials in
$x_2,x_3$. If $x_2^2x_3\notin L$, then every $x^\alpha\in L$ satisfies
$\alpha_1+\alpha_2\le1$, leaving only three possibilities.
Thus $L$ contains both $x_1x_3^2$ and $x_2^2x_3$, and hence all four
monomials in $D\cap D'$. By \eqref{eq:borel}, every monomial of $E(J)_3$
outside $D\cup D'$ lies strictly above $x_1x_2x_3$. Including any such
monomial would also require $x_1x_2x_3$, giving at least six elements.
The fifth monomial is therefore either $x_1x_2x_3$ or $x_2^3$.

For $u\in E(J)_3$, take $ux_4$ to be standard exactly when $u$
belongs to the chosen set.
By Proposition~\ref{prop:fiber}, each choice gives a strongly stable
ideal whose quotient has Hilbert function $g$ and strong Lefschetz
element $x_4$. Since $R_3/J$ has SLP, both quotients have $2$-SLP.
The monomial $x_1x_2x_3x_4$ is standard for the first ideal and
belongs to the second. By Lemma~\ref{lem:ginclass}, these are exactly
the two gins that occur for $g$ under $2$-SLP.
\end{example}

\Needspace{12\baselineskip}
\subsection{Beyond the quasi-symmetry hypotheses}
For a finite sequence of positive integers $g=(g_0,\ldots,g_e)$
with $g_0=1$ and $e\ge1$, put $P_i=\sum_{r=0}^i g_r$ for
$0\le i\le e$ and define
\begin{equation}\label{eq:symmetric-lift}
 \mathcal S(g)=(P_0,\ldots,P_e,P_{e-1},\ldots,P_0).
\end{equation}
Since $P_0=g_0$ and $P_i-P_{i-1}=g_i>0$ for $1\le i\le e$,
$\mathcal S(g)$ increases strictly to $P_e$ and then takes the
preceding values in reverse order. Thus $\mathcal S(g)$ is
quasi-symmetric and $\Delta\mathcal S(g)=g$.

Let $n\ge4$ and let $g$ be admissible in $n-1$ variables.
By Proposition~\ref{prop:admissibility}, $\Delta^{n-3}g$ is an
$O$-sequence. Since $\mathcal S(g)_1=1+g_1=n$ and
$\Delta^{n-2}\mathcal S(g)=\Delta^{n-3}g$, the same proposition
shows that $\mathcal S(g)$ is admissible in $n$ variables.
By Corollary~\ref{cor:reduction},
$\mathcal S(g)$ has the uniqueness property in $n$ variables
if and only if $g$ has it in $n-1$ variables.

\Needspace{6\baselineskip}
\begin{proposition}\label{prop:families}
For every \mbox{$n\ge4$}, there are infinitely many Hilbert functions $h$
admissible in $n$ variables such that $h$ has the uniqueness property,
$h^{[j]}$ is symmetric for every $0\le j<n-4$, and $h^{[n-4]}$ is
not quasi-symmetric.
\end{proposition}
\begin{proof}
Apply $\mathcal S$ exactly $n-4$ times to each $h^{(\ell)}$ from
Example~\ref{ex:four}. The resulting Hilbert function $h$ is
admissible in $n$ variables and has the uniqueness property.
Since $\Delta\mathcal S(g)=g$, for $0\le j\le n-4$ we have
\[
 h^{[j]}=\mathcal S^{\,n-4-j}(h^{(\ell)}),
\]
where $\mathcal S^0$ is the identity. For $j<n-4$, the right-hand
side is symmetric by \eqref{eq:symmetric-lift}. For $j=n-4$ it is
$h^{(\ell)}$, which is not quasi-symmetric. Distinct values of $\ell$
give distinct $h^{[n-4]}$ and therefore distinct $h$.
\end{proof}

\Needspace{9\baselineskip}
\begin{example}
Consider the two symmetric sequences
\[
 \begin{aligned}
 H&=(1,5,15,35,53,35,15,5,1),\\
 G&=(1,5,15,35,50,35,15,5,1).
 \end{aligned}
\]
Their positive differences are
\[
 \Delta H=(1,4,10,20,18),\qquad \Delta G=(1,4,10,20,15).
\]
By Example~\ref{ex:four}, $\Delta H$ and $\Delta G$ are admissible in
four variables. Formula~\eqref{eq:symmetric-lift} gives
$H=\mathcal S(\Delta H)$ and $G=\mathcal S(\Delta G)$, so $H$ and $G$
are admissible in five variables.
By Corollary~\ref{cor:reduction} and \eqref{eq:example-classification},
$H$ has the uniqueness property in five variables. Since $G$ is
quasi-symmetric, its section map \eqref{eq:section-map} is bijective
by \cite[Proposition~21(ii)]{HarimaWachi}. The two gins for $\Delta G$
in Example~\ref{ex:four} therefore give exactly two gins for $G$.
For both $H$ and $G$, no comparison is required at $j=0$, while the
comparison at $(j,q,t)=(1,3,4)$ holds for $H$ ($N=10$, $k=8$)
and fails for $G$ ($N=10$, $k=5$).
\end{example}

\Needspace{6\baselineskip}
\begin{example}\label{ex:five}
Let $h_+=(1,5,15,7)$ and $h_-=(1,5,15,8)$. Their common positive
difference $p=(1,4,10)$ counts all monomials of $R_4$ through degree
two, so the only possible section is $J=(x_1,\ldots,x_4)^3$.
For $r=3,4$, multiplication by each power of $x_r$ on
$R_r/(J\cap R_r)$ is injective in target degrees at most two and has
zero target otherwise. Thus $\B_4(p)=\{J\}$, and lifting $J$ by
Proposition~\ref{prop:fiber} shows that both functions are admissible
in five variables.

For $h\in\{h_+,h_-\}$, we have $\delta_0=3$ and $\delta_1=\delta_2=2$, so
no comparison is required at $j=1$. At $j=0$, $h_1<h_3<h_2$ gives
$(q,t)=(2,3)$ by Remark~\ref{rem:algorithm}. Here $N=10$ and
$k=h_3-5$, which equals $2$ for $h_+$ and $3$ for $h_-$.

In increasing degree reverse lexicographic order, $E(J)_2$ begins with
\[
 x_4^2,\quad x_3x_4,\quad x_2x_4,\quad x_1x_4,\quad x_3^2.
\]
Every monomial after the first two
lies above $x_3x_4$ in the Borel order, so $F=\{x_4^2,x_3x_4\}$ is
the unique lower set of size two. Outside $F$, the multiples of $x_4$
are all $\succeq x_2x_4$, and the monomials in $R_3$ are all
$\succeq x_3^2$. The third and fifth monomials have vectors
$\sigma_4(x_2x_4)=(0,1,1)$ and $\sigma_4(x_3^2)=(0,0,2)$, so
\eqref{eq:test} fails in the second coordinate for $k=3$.
These monomials are incomparable, and the only lower sets of size three
are $F\cup\{x_2x_4\}$ and $F\cup\{x_3^2\}$.
For $(q,t)\ne(2,3)$, \eqref{eq:rho-local} forces $D_{q,t}$ to be
empty or equal to $E(J)_q$. By
Proposition~\ref{prop:fiber} and Lemma~\ref{lem:ginclass},
exactly one gin occurs for $h_+$ and exactly two for $h_-$ among
five-variable quotients with $3$-SLP.
\end{example}

\Needspace{10\baselineskip}

\end{document}